\documentclass{amsart}
\usepackage{listings}
\usepackage{graphicx}
\usepackage{comment}
\newtheorem{theorem}{Theorem}[section]
\newtheorem{thm}[theorem]{Theorem}

\def\gcd{{\rm gcd}}
\def\mod{\; {\rm mod} \; }
\def\X{\mathcal{X}}
\def\Y{\mathcal{Y}}
\def\Z{\mathbb{Z}}

\title{A Multivariable Chinese Remainder Theorem}
\date{v1: January 27, 2005, v2: June 22, 2012}
\author{Oliver Knill}
\address{
        Department of Mathematics \\
        Harvard University \\
        Cambridge, MA, 02138\\
        }


\begin{document}

\maketitle

\begin{abstract}
Using an adaptation of Qin Jiushao's method from the 13th century, it is possible to 
prove that a system of linear modular equations 
$a_{i1} x_i + \cdots + a_{in} x_n = \vec{b}_i  \mod  \vec{m}_i, i=1,\dots ,n$
has integer solutions if $m_i>1$ are pairwise relatively prime and in each row, 
at least one matrix element $a_{ij}$ is relatively prime to $m_i$. 
The Chinese remainder theorem is the special case, where $A$ has only one column.
\end{abstract}

\section{The statement with proof}

Consider a linear system of equations $A \vec{x} = \vec{b}  \mod  \vec{m}$,
where $A$ is an integer $n \times n$ matrix and $\vec{b},\vec{m}$ are
integer vectors with coefficients $m_i>1$.

\begin{thm}[Multivariable CRT]
If $m_i$ are pairwise relatively prime and in each row, at least one matrix
element is relatively prime to $m_i$, then
$A \vec{x} = \vec{b}  \mod  \vec{m}$ has solutions for all $\vec{b}$.
There is a solution $\vec{x}$ in an $n$-dimensional
parallelepiped ${\mathcal X} = \Z_M^n/L$ of volume $M=m_1 \cdots  m_n$, where
$L$ is a lattice in $\Z_M^n$.
\end{thm}

\begin{proof}
The map $\phi: x \to A x  \mod  \vec{m}$ is a group homomorphism from the Abelian group
$X=\Z^n$ to the finite Abelian group $\Y = \Z_{m_1} \times \cdots \times \Z_{m_n}=\Y/L$,
where $L=(m_1 \Z) \times \cdots \times (m_n \Z)$ is a lattice subgroup of $\Y$.
The kernel of $\phi$ is a subgroup $L_A$ of $X$ and $\X= X/L_A$.
The image of $\phi$ is a subgroup of $\Y$.
By the first isomorphism theorem in group theory, the quotient group $\X$ and the image are isomorphic.
The kernel $L_A$ is a lattice in $X$ spanned by $n$ vectors $\vec{k}_1, \dots ,\vec{k}_n$. 
The map $\phi$ is injective on $\X$. By the Lagrange theorem in group
theory, there exist finitely many vectors $\vec{y}_i \in \Y$ such that
$\bigcup_{i=1}^{d(A)} A(\X) + \vec{y}_i = \Y$.
The problem is solvable for all $\vec{b}$
if and only if $d(A)=1$. For every $\vec{b}$, there exists then a unique integer vector
$\vec{x}$ in $\X$ such that $A \vec{x} = \vec{b}  \mod  \vec{m}$.
As in the usual CRT, we have a solution if each equation has a solution.
To construct a solution, pick matrix elements $a_{ij(k)}$
such that the $i$'th row is relatively prime to $m_i$. Let $\vec{e}_j$ denote the
standard basis in n-dimensional space. 
Consider a line $\vec{x}(t) = t \vec{e}_{j(1)}$ in $X$, where $t$ is an
integer. There exists an integer $t_1$ so that $\vec{x}(t)$ solves the first equation.
Now take the line $\vec{x}(t) =  t_1 \vec{e}_{j(1)} + t m_1 \vec{e}_{j(2)}$.
There is an integer $t_2$ so that $\vec{x}(t)$ solves the second equation.
This is possible because $m_1$ is relatively prime to $m_2$. Note that
$\vec{x}(t)$ still solves the first equation for all $t$. We have now a 
solution to two equations. Continue in the same way until the final solution
$\vec{x}(t) = \sum t_i (m_1 \dots m_i) \vec{e}_{i j(i)}$ is reached.
\end{proof}

{\bf Example:}
$$   \left[ \begin{array}{cc} 101 & 107 \\ 51 & 22 \end{array} \right] \left[ \begin{array}{c} x \\ y \end{array}  \right]
   = \left[ \begin{array}{c} 3 \\ 7 \end{array} \right] \; \mod \left[ \begin{array}{c} 117 \\ 71 \end{array} \right] \;  $$
is solved by $\left[ \begin{array}{c} x \\y \end{array} \right] = \left[ \begin{array}{c} 25 \\ 65 \end{array} \right]$. 
The lattice $L_A$ is spanned by $\left[ \begin{array}{c} 73 \\ 47 \end{array} \right],
\left[ \begin{array}{c} -82 \\ 61 \end{array} \right]$. \\

{\bf Remark.} The original paper of January 27, 2005 (google "multivariable chinese remainder")
had been written in the context of multidimensional Diophantine approximation 
and was part of a talk on April 11, 2005 in Dmitry Kleinbock's seminar at Brandeis. 
[I had at that time been interested in the Diophantine problem to find for $\theta,\alpha,\beta$ small integers $n,m$ such that 
$\theta + \alpha n + \beta m$ is close to an integer. In one dimensions this is 
achieved by continued fraction expansions \cite{Chinchin92,Hua}, 
but the problem is complex in two or higher dimensions. Solving
it effectively would lead to faster integer factorization algorithms.]
Referees at that time found the paper too elementary. While this is probably true, I feel even 
after 7 years and more literature research, that the result might well have been overlooked.
Additionally, the multivariable CRT can serve as an exercise in algebra or spice up an exposition 
about the traditional CRT. The story also has a historical angle when looking for 
the origin of solving systems of linear equations with integer solutions. There was some controversy for example whether 
Nicomachus has known anything about the CRT even so there is 
much evidence against it \cite{Libbrecht}. Historically, it appears now certain that 
the method to solve the multivariable CRT is due to the mathematician Qin in the 13th century, 
an algorithm which in modern language would be considered a special case and precursor
for the Schreier-Sims algorithm \cite{Baumslag} 
in computational group theory, an algorithm which is naturally used by
everybody who solves puzzles whether it is the simple 15 puzzle or the more challenging Rubik cube \cite{Joyner}. \\

The following changes were done for this update: the statement, the proof and an example are
stated initially, some figures are gone, the text is streamlined and examples and 
remarks are separated. Remarks 10)-16) as well as more references like 
\cite{Koshy,Swetz,Libbrecht,Kangsheng,Gauss,YanShiran,Katz2011,Stillwell,Dauben,Martzloff,Fiol,GGR} 
as well as some Mathematica code was added. A new literature search revealed
\cite{GGR} from which one can deduce our theorem, but it looks considerably less elementary.
More digging in sources revealed that the proof is close to Qin Jiushao's ``method of finding one".
This algorithm from the 13th century is especially remarkable because Qin did not have group theory 
nor even the notion of prime numbers at that time.

\section{Historical background}

The Chinese remainder theorem (CRT) is one of the oldest theorems 
in mathematics. It has been used to calculate calendars as early 
as the first century AD \cite{dicksonII,sheng88}. 
The earliest recorded instance of work with indeterminate equations in 
China can be found in the 'Chiu-Chang Suan Shu', the ``Nine Chapters on the mathematical art",
where a system of four equations with
five unknowns appears \cite{Swetz}. This text is also an early source for Gaussian elimination \cite{Stillwell,Katz2011}.
The mathematician Sun-Tsu (also Sun Zi), in the Chinese work 'Sunzi Suanjiing' (Sun Tzu Suan Ching)
which translates to either ``Master Sun's Mathematical Manual" or ``Sun-Tzu's Calculation Classic"
from the third century considered the problem to find integral solutions to
\begin{eqnarray*} 
         x &=& 2  \mod  3 \\
         x &=& 3  \mod  5 \\ 
         x &=& 2  \mod  7 \; . 
\end{eqnarray*} 
This example has the solution $x=23$ (see also \cite{Dauben}). It has been reported in \cite{dicksonII} 
\footnote{Dickson references Y. Mikami, Abh. Geschichte Math, 30, 1912, p.32.
The connection of Nicomachus with the CRT is disputed in \cite{Libbrecht}. 
We also could not find the example in \cite{Nicomachus} but
other textbooks of Nicomachus are only referred to by other sources.} 
to appear also in a textbook of Nicomachus of Gerasa in the first century. The consensus is
today however, Sun Zi's text is indeed the first known occurrence of the CRT and that there is no Nicomachus
connection. Unfortunately, 'Sunzi Suanjing' is hard to date. The average of all estimates 
points towards 250AD \cite{ConnorRobertsonSunzi} but it could be dated as late as the 4th 
century \cite{Libbrecht}. \\

Mathematics earlier than that is probably void of the CRT. It is not a topic of
'Suan Shu Shu' for example, an ancient Chinese collection of writings on
bamboo strips \cite{Dauben} which is an anonymous text from about 200 BC 
and which does not contain linear algebra yet. 
Calendars were the presumably the major motivation for the CRT \cite{Kangsheng}:
{\it Congruences of first degree were necessary to calculate calendars in ancient China as early as the 2nd 
century BC. Subsequently, in making the Jingchu calendar (237,AD), the astronomers defined 
Shangyuan as the starting point of the calendar.}
``Master Sun's math manual'' is now considered the earliest source of the CRT and 
the 'Shushu Jiuzhang' = ``Mathematical Treatise in Nine Sections" in 1247 the earliest description of a
solution algorithm. 
\footnote{This 13th century text should not be confused with the much older arithmetic textbook 
"Nine Chapters on the Mathematical Art".} According to 
\cite{YanShiran}, Qin Jiushao called his technique ``method of finding one",
which achieved his goal without using concepts like prime number of prime factors.
While nine problems in that text were exercises without applications,
there was one problem dealing with calendar applications \cite{Martzloff}.
The Mathematics in that work is a major topic in the thesis \cite{Libbrecht}
of the sinologist Ulrich Libbrecht. \\ 

The development of the CRT from the fourth to the 16'th century is fascinating and multi-cultural. 
In chapter 14 of \cite{Libbrecht} we can read: 
{\it In India there were Brahmagupta (ca 625) and Bhaskara (12th century),
who developed the kuttaka method. In the Islamic world, Ibn al-Haitham treats this kind of 
problem and he may have influenced Leonardo Pisano (Fibonacci) in Italy. After the 
thirteenth century we do not find much further investigation in China, India or the Islamic world. But
from the fifteenth century on there is a marked increase in European research, 
which reached its apogee in the studies of Lagrange, Euler and Gauss.} 
Also Japanese mathematicians were involved: from \cite{Kangsheng}: {\it The Japanese mathematician Seki Takakazua wrote 
``Kwatsuyo sampo" (Essential Algorithm) in 1683, the second chapter of which, Sho yukujutu su, 
deals with some algorithms corresponding to Qin's work.}. The last reference is of course to
Qin Jiushao (=Chiu-Shao) the author of ``Shushu Jiuzhang". \\

Linear congruences of more variables must have appeared only later. We do not count in 
examples like $5x+3y+z/3=100,x+y+z=100$ which occur in 
Zhanag Quijian's Mathematical manual of 475 \cite{Katz2011} 
with two equations of three variables. This can be considered a case for single variables
because a substitution leads to a modular equation in one variable.
Dickson \cite{dicksonII} gives as the first reference Sch\"onemann, who considered in 
the year 1839 equations of the form
$a_1 x_1 + \cdots + a_n x_n = 0  \mod  p$,
where $p$ is a prime. It was probably Gauss \cite{Gauss}, who first looked at modular systems of 
$n$ linear equations of $n$ unknowns. The case of equal moduli definitely dominates. 
The case of different moduli is less clear since it was usually reduced to the case of equal moduli: 
George Mathews noted in \cite{mathews1892} that a system of linear equations 
$A \vec{x} = \vec{b}   \mod  \vec{m}$ can be reduced to 
a system $B \vec{x} = \vec{a}  \mod  m$, where 
$m= {\rm lcm}(m_1,\dots ,m_n)$.  For example, the system
\begin{eqnarray*} 
         x+y &=& 1  \mod  3 \\
         x-y &=& 2  \mod  5 
\end{eqnarray*} 
which has solution $x=3,y=1$ is equivalent to 
\begin{eqnarray*} 
         5x+5y &=& 5  \mod  15 \\
         3x-3y &=& 6  \mod  15  \; . 
\end{eqnarray*}
However, since many results and methods developed for a single moduli do not work
- like row reduction or the inversion by Cramer's formula -
there is not much gained with such a reduction. The Mathews reduction does also 
not allow to use the multivariable CRT generalization proven here. The reduction 
even does not help to solve the single variable CRT! Actually, the multivariable CRT
we study here is in nature closer to the one variable CRT than to linear algebra.\\

Gauss treated in his "disquisitiones arithmetica" (see \cite{Gauss} page 29) of 1801
systems of linear congruences but also with equal moduli. He considered in particular the system
\begin{eqnarray*} 
         3x+5y +  z &=& 4  \mod  12 \\
         2x+3y + 2z &=& 7  \mod  12 \\
         5x+ y + 3z &=& 6  \mod  12 
\end{eqnarray*}
which has the four solutions $(2,11,3),(5,11,6),(8,11,9),(11,11,0)$ 
in ${\bf Z}_{12}^3$. 
The discrete parallelepiped spanned by $(3,0,3),(12,0,0),(0,12,0)$ is mapped by the linear
map $A$ bijectively to a proper subset of $\Z_{12}^3$. Indeed, 
the matrix $A$ over the ring $\Z_{12}$ is not invertible because
${\rm det}(A)=4$ is not invertible in $\Z_{12}$. 
In Gauss example, there is a parallelepiped in $\Z_{12}^3$ which is
mapped onto a proper subset of $\Z_{12}^3$ by the transformation $A \vec{x}  \mod  12$. 
For the same matrix $A$, only one forth of all vectors $\vec{b}$ in $\Z_{12}^3$ allow
that $A \vec{x} = \vec{b}  \mod  12$ can be solved. In that case,
there are four solutions. 
Elimination was used by 
Gauss also as a method to solve such linear systems of Diophantine equations: subtracting 
the last row from the sum of the first two gives $7 y = 5  \mod  12$ or
$y=11$. We end up with the system 
\begin{eqnarray*} 
         3x+ z  &=& 9  \mod  12 \\
         5x+ 3z &=& 7  \mod  12  \; . 
\end{eqnarray*}
Eliminating $x$ gives $4z=0  \mod  12$ or $z = 0  \mod  3$
which leads to the $4$ solutions $z=0,3,6,9$. In each case, the solution $x$ is determined. 
H.J.S. Smith \cite{dicksonII} noted in 1859 that if all moduli are the same $m$ and ${\rm det}(A)$ is 
relatively prime to $m$, then $A \vec{x} = \vec{b}  \mod  m$ has a unique 
solution in the module $\Z_m^n$ over the ring $\Z_m$.
Also Cramer's rule (from 1750 evenso in the context of real numbers) gives the explicit solution
$\vec{x}_i = {\rm det}(A_{\vec{b},i}) {\rm det}(A)^{-1}$
in which the determinant ${\rm det}(A)$ is inverted in $\Z_m$ and 
$A_{\vec{b},i}$ is the matrix in which the $i$'th column had been
replaced by $\vec{b}$. Smith had noted first that ${\rm det}(A)$ must 
be prime to $m$. For integer matrix arithmetic in number theory, see \cite{Hua}.  \\

Systems of linear modular equations definitely have been treated in the 18'th century,
when all moduli $m_i$ are equal. The general
case with different $m_i$ can be reduced to this case when all the moduli are
all powers of prime numbers with the equivalence of each equation 
$a_1 x_1 + \cdots+ a_n x_n = b \mod q_1^{k_1} \cdots q_l^{k_l}$ to 
\begin{eqnarray*}
    a_1 x_1 + \cdots + a_n x_n &=& b \mod q_1^{k_1} \\
        &\dots&  \\
    a_1 x_1 + \cdots + a_n x_n &=& b \mod q_l^{k_l} \; .
\end{eqnarray*}
As in the CRT, we can not do row-reduction with different moduli in general so that this is not
a standard linear algebra problem any more. 
As the Mathew reduction has shown, the general case can also be reduced to the case when all moduli 
are equal but methods which worked before can no more be applied then.
For example, the determinant of the new matrix is zero in $\Z_m$. \\

As in linear algebra, complexity problems of solving$A \vec{x} = \vec{b}  \mod  \vec{p}$ are far from trivial. 
Beside the aim to find the structure of the solutions of a system of modular linear equations, 
there is the computational task to find solutions and a minimal area parallelepiped in $\Z_M^n$ on which
$A$ is injective as a map to $\Z_{m_1} \times \cdots  \times \Z_{m_n}$. 
How many computation steps are needed to
decide whether a system has a solution and how many steps are required to
find it? The question is addressed in \cite{DoSt01}, where the problem is dealt with
the method of quantifier elimination in discretely valued fields.  \\

Our approach here is elementary like the single variable CRT and generalizes Qin's approach to
the usual CRT, in which solutions can be found in $\{0, \dots ,M-1 \; \}$, which can 
also be interpreted as a parallelepiped of length
$M=m_1 m_2 \cdots m_n$ and width dimensions of length $1$. \\

We still do not know the best way to find an optimal kernel (LLL helps a lot but 
is not always optimal) and decide effectively, when a general system
$A \vec{x} = \vec{b}   \mod  \vec{m}$ has a solution and
when not. Our theorem only gives a sufficient condition. 
The efficiency part is especially relevant in 
cryptological context like in lattice attacks \cite{how01}, where
one tries to reconstruct the keys from several messages.

\section{Examples}

We look now at a few examples of systems of $n=2$ equations $A \vec{x} = \vec{b}  \mod  \vec{m}$,
where $\vec{m}=(p,q)$ has the property that $p,q$
are relatively prime. Unlike in the situation $\vec{m}=(p,p)$ with prime $p$,
where the solution can be found in the fixed algebra over the finite field $\Z_p$, it
does now not matter in general, how singular the matrix $A$ is.
The decision known from linear algebra about the existence of solutions,
unique solvabilty or non-solvabilty has still to be made: \\

{\bf Example 1)}
\begin{eqnarray*} 
         x+y &=& 1  \mod  3 \\
         x-y &=& 2  \mod  5 \; . 
\end{eqnarray*}  
To a given solution like $\vec{x} = (3,1)$, we can add
solutions of the homogeneous equation $A \vec{x}_0 = \vec{0}$
like $(2,7),(3,3),(-1,4),(1,11)$. This is an example, where
solutions exist for all vectors $\vec{b}$. The curve
$\vec{x}(t)=(3t,t) \mod \; p$ reduces the problem to the 
single variable CRT case
\begin{eqnarray*} 
         4t &=& 1  \mod  3 \\
         2t &=& 2  \mod  5 
\end{eqnarray*}
which always can be solved for $t$. \\

{\bf Example 2)}
\begin{eqnarray*} 
        2x  + 3y   &=& 6  \mod  7  \\
       -3x  - 9y   &=& 3  \mod  12  \; . 
\end{eqnarray*}
This is an example, where the existence of integer solution $(x,y)$ depends
on the vector $\vec{b}$. The above example has a solution. The system
\begin{eqnarray*} 
        2x  + 3y   &=& 1  \mod  7  \\
       -3x  - 9y   &=& 1  \mod  12 
\end{eqnarray*}
has no solution. In the set $\Z_7 \times \Z_{12}$ with $84$ elements, we count 
28 vectors $\vec{b}$ for which there is a solution and $56$ elements, for 
which there is no solution.  \\


{\bf Example 3)}
\begin{eqnarray*} 
        6x  - 4y   &=& 7  \mod  7  \\
       10x  - 5y   &=& 1  \mod  5  \; . 
\end{eqnarray*} 
There is no solution because the second equation reads 
$0=1$ modulo $5$. However, for a different $\vec{b}$ like
\begin{eqnarray*} 
        6x  - 4y   &=& 2  \mod  7  \\
       10x  - 5y   &=& 5  \mod  5  \; , 
\end{eqnarray*}
we have a solution $\vec{x}=(1,1)$. In the set $\Z_7 \times \Z_{5}$ with $35$ 
elements, only $7$ vectors $\vec{b}$ give a system with a solution.  \\


{\bf Example 4)} The system
\begin{eqnarray*} 
         x+y &=& 1  \mod  3 \\
         x+y &=& 2  \mod  5  \; 
\end{eqnarray*}
can be reduced to a case of the CRT case:
\begin{eqnarray*} 
         z &=& 1  \mod  3 \\
         z &=& 2  \mod  5  \; 
\end{eqnarray*}
and is solved for $z=7$. In the set $\Z_3 \times \Z_5$ with 15 elements, 
every vector $\vec{b}$ has a unique solution $z$. The original system has now
solutions like $\vec{x}=(1,6)$ or $\vec{x}=(2,5)$.  \\


{\bf Example 5)} 
The size of the lattice $L$ in $\Z_M^n$ can vary when $\vec{m}$ is fixed. Here
is a case with a relatively narrow lattice spanned by the vectors $(1,-3), (43,14)$:


\begin{eqnarray*} 
        6x  - 2y  &=& 0  \mod  11  \\
       11x  - 5y  &=& 0  \mod  13  \; . 
\end{eqnarray*}
The extreme case is the CRT case, where the
lattice has dimensions $143 \times 1$: 
\begin{eqnarray*} 
        6x  - 3y  &=& 0  \mod  11  \\
       12x  - 6y  &=& 0  \mod  13 \; . 
\end{eqnarray*}

Next, we look now at examples, where the moduli $m_i$ are not necessarily pairwise
prime:  \\

{\bf Example 6)} This example is a case for linear algebra.
If $\vec{m} = (m_1\dots,m_n)=(p,\dots,p)$, where $p$ is a prime number, we have 
a linear system of equations over the finite field $F_p$. 
This is a problem of linear algebra, where solutions can be found 
by Gaussian elimination or by inverting the matrix. If the determinant 
of $A$ is nonzero in the field $\Z_p$, then $A^{-1}$ exists and 
$x=A^{-1} y$. For example, with $p=11$, solving 
$A \vec{x} = \vec{b}  \mod  \vec{m}$:
$$ \left[ \begin{array}{ccc} 
 2 & 1 & 2 \\
 1 & 2 & 9 \\
 1 & 2 & 7 
\end{array} \right]  \left[ \begin{array}{c} x \\ y \\ z \end{array} \right]
= \left[ \begin{array}{c} 1 \\ 2 \\ 3 \end{array} \right]  \mod  11  $$
is done in the same way as over the field of real numbers. The determinant is
$5$ modulo $p=11$ so that the matrix is invertible over $F_p$. The inverse of 
$A$ in $F_p$ is 
$A^{-1} = \left[ \begin{array}{ccc} 
      8 & 6 & 1 \\
      7 & 9 & 10 \\
      0 & 6 & 5 
                 \end{array} \right]$ and $A^{-1} \vec{b} 
    = \left[ \begin{array}{c} 1 \\ 0 \\ 5 \end{array} \right]$. 
Indeed $\vec{x} = \left[ \begin{array}{c} 1 \\0 \\ 5 \end{array} \right]$ 
solves the original system of equations. \\

{\bf Example 7)}:
If only one column is nonzero, we have the Chinese reminder theorem.
If the matrix $A$ has only one nonzero column, we are in the CRT situation.
This problem was  considered 2000 years ago and was given its final form by Euler. 
For example, 
$$  \left[ \begin{array}{ccccc} 
           0 & 2 & 0 & 0 & 0 \\
           0 & 3 & 0 & 0 & 0 \\
           0 & 1 & 0 & 0 & 0 \\
           0 & 9 & 0 & 0 & 0 \\
                       \end{array} \right] 
    \left[ \begin{array}{c} x_1 \\ x_2 \\ x_3 \\ x_4 \end{array} \right] 
   = \left[ \begin{array}{c} 5 \\ 8 \\ 11 \\ 9 \end{array} \right] 
    \mod  
    \left[ \begin{array}{c} 3 \\ 11 \\ 7 \\ 13 \end{array} \right]  $$
is equivalent to 
\begin{eqnarray*}
 2 x&=&5   \mod  3   \\
 3 x&=&8   \mod   11  \\
   x&=&11  \mod  7  \\
 9 x&=&9   \mod  13 \;  .
\end{eqnarray*} 
Also the original CRT problem
$a_i x = b_i  \mod  m_i$ can be solved
in a geometric language: with an integer "time" parameter $t$ and the
"velocity" $\vec{v}=(v_1,...,v_n)$,  the parameterized curve
$\vec{r}(t) = t \vec{v} \;  \mod \; \vec{m}$ is a line
on the "discrete torus" $\Y = \Z_{m_1} \times \cdots \times \Z_{m_n}$.
It covers the entire torus if the integers $m_i$ are pairwise relatively prime 
and $a_i \neq 0  \mod  m_i$. 
One can solve the task of hitting a specific point $\vec{b}$ on the torus by 
solving the first equation $v_1 x_1 = b_1  \mod  m_1$, then consider 
the curve $v_1 (x_1 + m_1 t)$, reducing the problem to a similar problem 
in one dimension less. Proceeding like this leads to the solution.
The solution for the CRT was easy to find, because the group was Abelian.
The strategy to retreat in larger and larger centralizer subgroups is also
the key to navigate around in non-Abelian finite groups like
``Rubik" type puzzles, where one first fixes a part of the cube and then tries to 
construct words in the finitely presented group which fixed that subgroup. 
It is a natural idea which puzzle-solvers without mathematical 
training come up with. By the way, also the Gaussian elimination process 
is an incarnation of this principle.  \\

{\bf Example 8)} Here is a case where we have independent equations.
If $A$ is a diagonal matrix we have $n$ independent equations of the form 
$a_i x_i = b_i  \mod   m_i$. Solutions exist if $\gcd(a_i,m_i)=1$
for all $i$.  If $\gcd(a,m)>1$ like $a=3, p=6$, there are no solutions
of $3 x = 2  \mod  6$ as can be seen by inspecting 
the equation modulo $3$. Example: 
$$  A \vec{x} = \left[ \begin{array}{ccccc} 
           2 & 0 & 0 & 0 \\
           0 & 3 & 0 & 0 \\
           0 & 0 & 5 & 0 \\
           0 & 0 & 0 & 7 \\
                       \end{array} \right] 
    \left[ \begin{array}{c} x_1 \\ x_2 \\ x_3 \\ x_4 \end{array} \right] 
   = \left[ \begin{array}{c} 5 \\ 8 \\ 11 \\ 9 \end{array} \right] 
    \mod  
    \left[ \begin{array}{c} 3 \\ 11 \\ 7 \\ 13 \end{array} \right]  \; . $$

{\bf Example 9)} Here is a case, where row reduction works.
If $A$ is upper triangular or lower triangular matrix, the system can be 
solved by successively solving systems $a_i x = b_i  \mod   m_i$. 
Again, we have solutions if $\gcd(a_i,m_i)=1$ for all $i$. 
$$  A \vec{x} = \left[ \begin{array}{ccccc} 
           2 & 1 & 1 & 0 \\
           0 & 3 & 2 & 1 \\
           0 & 0 & 5 & 1 \\
           0 & 0 & 0 & 7 \\
                       \end{array} \right] 
    \left[ \begin{array}{c} x_1 \\ x_2 \\ x_3 \\ x_4 \end{array} \right] 
   = \left[ \begin{array}{c} 5 \\ 8 \\ 11 \\ 9 \end{array} \right] 
    \mod  
    \left[ \begin{array}{c} 3 \\ 11 \\ 7 \\ 13 \end{array} \right]  \; . $$

{\bf Example 10)} This is an example, when $A$ is modular.
If $A^{-1}$ has only integer entries, solutions can be obtained directly 
with the formula $x = A^{-1} \vec{b}$ in $\Z^n$. This works if $A$ is modular 
that is if $A$ has determinant $1$ or $-1$:
$$A \vec{x} = \left[ \begin{array}{ccc} 
 5 & 3 & 4 \\
 1 & 4 & 1 \\
 5 & 2 & 4 
\end{array} \right]  \left[ \begin{array}{c} x \\ y \\ z \end{array} \right]
= \left[ \begin{array}{c} 1 \\ 2 \\ 3 \end{array} \right]  \mod  
\left[ \begin{array}{c} 5 \\ 7 \\ 11 \end{array} \right]
= \vec{b}  \mod  \vec{m} \; . $$
We get 
$\left[ \begin{array}{c} x \\ y \\ z \end{array} \right]
= A^{-1}  \left[ \begin{array}{c} 1 \\ 2 \\ 3 \end{array} \right] 
= \left[ \begin{array}{ccc} 
14 &-4 & -13 \\
 1 & 0 & -1 \\
-18& 5 & 17 
\end{array} \right] \left[ \begin{array}{c} 1 \\ 2 \\ 3 \end{array} \right]  
=  \left[ \begin{array}{c}-33 \\ -2 \\ 43 \end{array} \right]$.  


{\bf Example 11)}. For
\begin{eqnarray*}
x+3y+z  &=& 1  \mod  8 \\
4x+y+5z &=& 7  \mod  8 \\
2x+2y+z &=& 3  \mod  8 \; ,
\end{eqnarray*}
Gauss gives the solution $x=6,y=4, z=7  \mod  8$ . Indeed, we would write today modulo 8,                
$$ A=\left[ \begin{array}{ccc}
                  1 & 3 & 1 \\
                  4 & 1 & 5 \\
                  2 & 2 & 1
                 \end{array} \right]
   A^{-1} = \left[
                  \begin{array}{ccc}
                   1 & 1 & 2 \\
                   2 & 1 & 1 \\
                   2 & 4 & 3
                  \end{array}
                  \right], A^{-1}.\left[ \begin{array}{c} 1\\ 7 \\ 3 \end{array} \right]  
            = \left[ \begin{array}{c} 6 \\ 4 \\ 7 \end{array} \right] \; .  $$
On a more curious side, we could rewrite the original equations as
\begin{eqnarray*}
x+z+8u  = 1  \mod  3 \\
4x+y+8v = 7  \mod  5 \\
2x+z+8w = 3  \mod  2
\end{eqnarray*}
and apply the multivariable CRT to see that there is a solution.   

\section{Remarks}

{\bf 1)} If all moduli $m_i$ are equal to a prime $m=p$, the problem can be solved
using linear algebra over the finite field $F_p$. As noted first
150 years ago by Smith, if $m$ is not prime, but the determinant of the matrix $A$ is invertible in the
ring $\Z_m$, then the problem can be solved for all $\vec{b}$. \\
{\bf 2)} If $A$ has only one nonzero column, the problem is the
CRT, one of the first topics which appears in any introduction to number theory. 
Also if there is a column $A_{ij}$ with fixed $j$ for which $\gcd(A_{ij},m_i)=1$, then 
we can set $x_1,\dots,x_{j-1},x_{j+1},\dots,x_n=0$ and solve for $x_j$ with the one dimensional
CRT. \\
{\bf 3)} The lattice $L$ is not unique in general. For example, if the lattice spanned
by $\vec{v}_1,\dots,\vec{v}_n$, then it is also spanned by
$\vec{v}_1+\vec{v}_2,\vec{v}_2,\dots ,\vec{v}_n$ and the volume
is the same. \\
{\bf 4)} The multivariable CRT is sharp in the sense
that the two conditions for solvabilty are necessary in general, as examples have shown. 
As examples with equal moduli show other conditions for solvability exist.
The Matthew trick sometimes allows linear algebra methods, but not in general
because  matrix might have determinant $0$ or even not be square.
Already the single variable CRT can not be solved with linear algebra alone. \\
{\bf 5)} The parallelepiped can be very long. An extreme case is the CRT
situation, where it has length $M=m_1 m_2 \cdots m_n$ and all
other widths are $1$.  \\
{\bf 6)} It would be useful to quantize how large the
diameter of the parallelepiped is. If $A$ is unimodular, the eigenvalues of $A$ are
relevant.  \\
{\bf 7)} A modern algebraic formulation of the single variable CRT is that for
pairwise co-prime elements $m_1,\dots ,m_n$
in a principal ideal domain $R$, the map
$x  \mod  M \to (x  \mod  m_1, \dots , x  \mod  m_n)$
is an isomorphism between the rings
$R/(m_1 R) \times R/(m_n R)$ and $R/(M R)$.
Using the same language, the multivariable CRT can be restated that if $R$ is a principal
ideal domain and a ring homomorphism $A: R^n \to R^n$, for which the $i'th$ row of $A$
is not zero in $R/(q_i R)$ with factors $q_i>1$ of $m_i$, there is a lattice
$L$ in $R^n$ such that $A$ is a ring isomorphism
between $R^n/L$ and $R/(m_1 R) \times \cdots \times R/(m_n R)$.
When seen in such an algebraic frame work, the result is quite transparent
and might be "well known" in the sens that the multivariable CRT could well have entered as a homework
in an algebra text book, but we were unable to locate such a place yet. Also a
search through number theory text books could not reveal the statement
of the multivariable CRT.  \\
{\bf 8)} While the problem of {\bf systems} of linear modular equations
$A \vec{x} = \vec{b}  \mod  \vec{m}$ with different moduli
$m_i$ studied here certainly is elementary, the lack of linear algebra and group theory
two thousand years ago could explain why it had not been studied early.
The problem has the CRT as a special case and must in general 
be understood and solved without linear algebra. Indeed,
one of the proofs of the CRT essentially goes over to the multivariable CRT.
But the constructive aspect of finding $L$ and effectively inverting $\phi$
is interesting and much more difficult than in the special case of the CRT.   \\
{\bf 9)} There is unique solution to systems of modular equations
if and only if there is a line
$A (t \vec{v}) \mod \vec{m}$ which covers the entire torus
$\Y = \Z_{m_1} \times \cdots \times \Z_{m_n}$. If $\vec{v}$ is known,
then it reduces the multivariable CRT problem to a CRT problem.  \\
{\bf 10)} It is no restriction of generality to assume the matrix $A$ to be square. 
If we have less variables, we can add some zero column vectors and dummy variables which
will be set to zero. If we have more variables, we can duplicate some of the equations.
Both of these "completions" do not change anything in the theorem. \\
{\bf 11)} Systems of modular equations have either a unique solution,
no solution or finitely many solutions. In the third case, the number
of solutions is a factor of $M=m_1 \cdot \cdots \cdot m_n$.  \\
{\bf 12)} Any system linear modular equations can be written as a linear system $B x =y$ with
one variable more. For example,
\begin{eqnarray*}
2x_1+3x_2&=&1  \mod  5 \\
3x_1+5x_2&=&1  \mod  7
\end{eqnarray*}
can be written as
\begin{eqnarray*}
2x_1+3x_2+5x_3 &=&1 \\
3x_1+5x_2+7x_3 &=&1  \; .
\end{eqnarray*}
{\bf 13)} An important case is when we have only one equation $Ax=y$ like
$$ 3x+5y+7z=11  \; . $$
Then the system is solvable if and only if $\gcd(A_{11},\dots,A_{1n})$
divides $y$. This is a central result in linear Diophantine equations
(see e.g. \cite{Andreescu} Theorem 2.1.2). Note that our theorem covers
only the 'if' part here: if we rewrite the system as a modular equation like
$$ 3x+5y=11  \mod  7 $$
and one of the coefficients has no common divisor with $7$, then
$\gcd(A_{11},\dots,A_{1n})=1$. \\
{\bf 14)} Not every linear Diophantine system $Bx=y$ can be rewritten
as a modular system. The book \cite{Andreescu}
mentions a problem from the 18th international math olympiad:
{\it Show that $A x=0$ be a linear system of equations with
$A_{ij} \in \{-1,0,1\}, 1 \leq i \leq p, 1 \leq j \leq 2p$
has a nonzero integer solution vector $x$ with $|x_j| \leq q$.} \\
{\bf 15)} If we write down a random system of linear modular equations 
$A \vec{x} = \vec{b} \; \mod\; \vec{m}$ like taking random integers $\{0,1, \dots ,n\}$
in each entry of the matrix, vector $\vec{b}$ and $\vec{m}$. What is the chance to have
a solution? It is well known that the probability of two numbers to be coprime is asymptotically
$1/\zeta(2)=6/\pi^2 \sim 0.61...$. Thus the condition that $p_i$ is not coprime to any of the 
row entries $A_{ij}$ has probability $(1-1/\zeta(2))^n$ and the condition to have this in one row
is bound above by $n (1-1/\zeta(2))^n$ which goes to zero. 
The only relevant condition asymptotically is therefore the second condition that all the $m_i$ are 
pairwise prime. The probability of the vector $\vec{m}$ to be coprime is $1/\zeta(n)$ which goes to $1$ 
exponentially fast, but the probability to be pairwise coprime goes to zero. 
Thus we can only say that conditioned to the pairwise coprimality assumption of the $m_i$, 
a random linear modular system asymptotically has a solution almost surely. \\
{\bf 16)} A different generalization of the CRM theorem where the concept of congruence is 
generalized can be found in \cite{Fiol}. 
In \cite{GGR}, the CRT has been generalized using a more general group context. The authors
apply the theory to systems $Ax=b \; \mod \vec{m}$ in section 2 (page 1205 of the paper). 

\section{More about the proof}

Row operations as used in Gaussian elimination are not in general permitted
to solve the problem $A \vec{x} = \vec{b}  \mod  \vec{m}$ 
because each row is an equation in a different ring of integers. 
But the geometric solution of the CRT can be generalized to solve the general 
case as well as to locate {\bf small} solution vectors.  \\

Let us prove the multivariable CRT in more detail as in the introduction.
Assume $\gcd(m_i,m_j)=1$ for all $i \neq j$ 
and that for all $i=1,\dots,n$, there exists $j$ such that $\gcd(a_{ij},m_i)=1$.
We show that there is a solution $\vec{x}$ to the linear system 
$A \vec{x} = \vec{b}  \mod  \vec{m}$ for all 
$\vec{b}$. We also have to prove that the solution $\vec{x}$ is unique in a parallelepiped spanned by 
$n$ vectors. This parallelepiped contains $M=m_1 m_2 \cdots m_n$ lattice points.  \\

{\bf I. Existence}. \\
We have seen that $\phi: x \to A x  \mod  \vec{m}$ is a group homomorphism from 
$X=\Z^n$ to the finite group $\Y = \Z_{m_1} \times \cdots \times \Z_{m_n}=\Y/L$.
We can think of $\Y$ as a discrete torus
with $M=m_1 \cdot m_2 \dots \cdot m_n$ lattice points. We can think of the order $M$
of the group as the "volume" of the torus $\Y$. 
${\rm ker}(\phi)$ is a lattice $L_A$ satisfying  $\X = X/L_A$ and
${\rm im}(\phi)$ is a subgroup of $\Y$. The quotient group $\X$ and the image are isomorphic. 
The kernel $L_A$ is a lattice in $X$ spanned by $n$ vectors $\vec{k}_1, \dots
,\vec{k}_n$. We think of the quotient $\X = X/L_A$ as a "discrete torus" with
"volume" $|\X|$.
Because $\phi$ is injective on $\X$, there exist vectors $\vec{y}_i \in \Y$ such that 
$\bigcup_{i=1}^{d(A)} A(\X) + \vec{y}_i = \Y$ and 
$d(A) {\rm vol}(\X) = {\rm vol}(\Y)$. 
If $d(A)=1$ the problem is solvable: for $\vec{b}$, there exists a unique integer vector
$\vec{x}$ in $\X$ such that $A \vec{x} = \vec{b}  \mod  \vec{m}$. \\

\begin{figure}
\scalebox{0.25}{\includegraphics{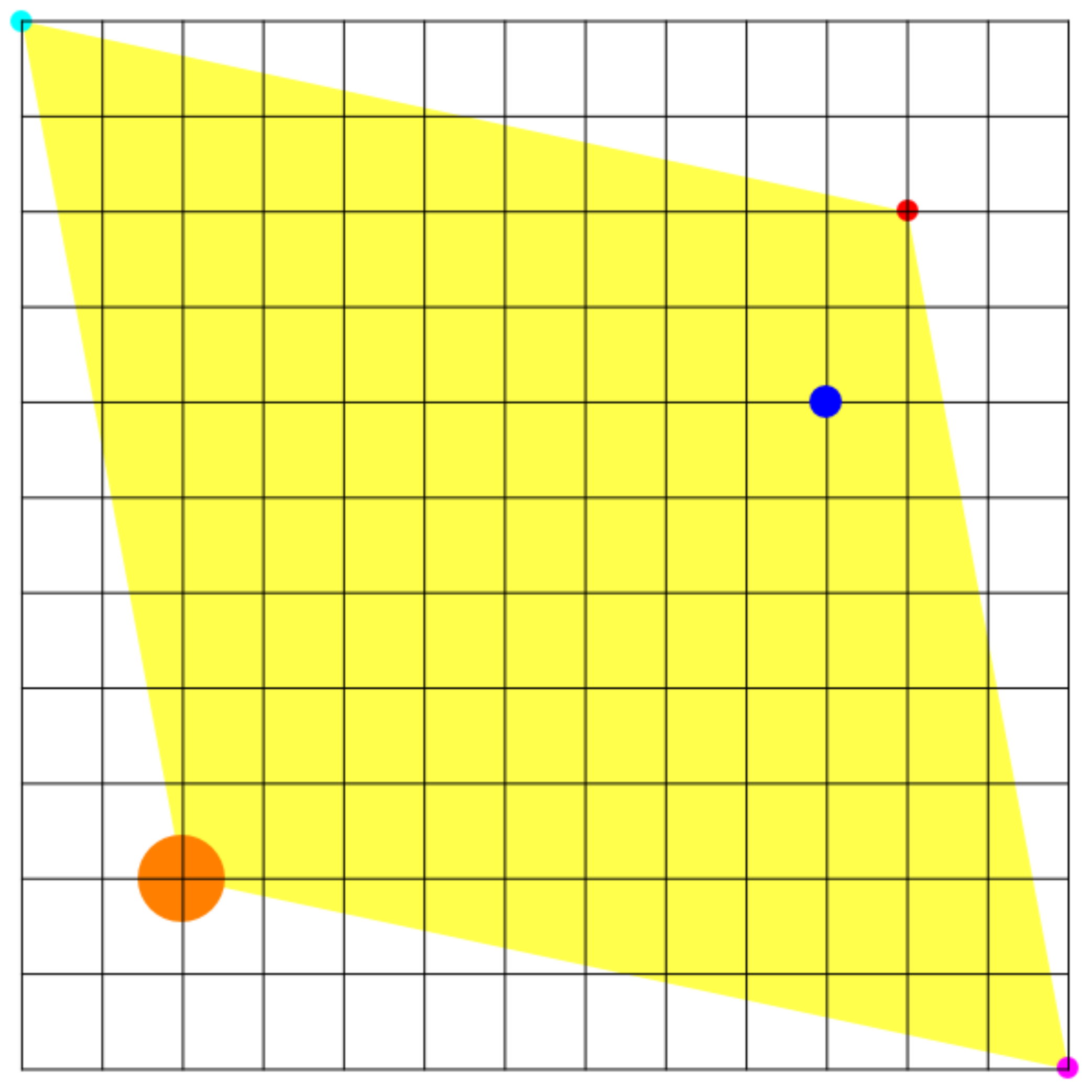}} 
\scalebox{0.25}{\includegraphics{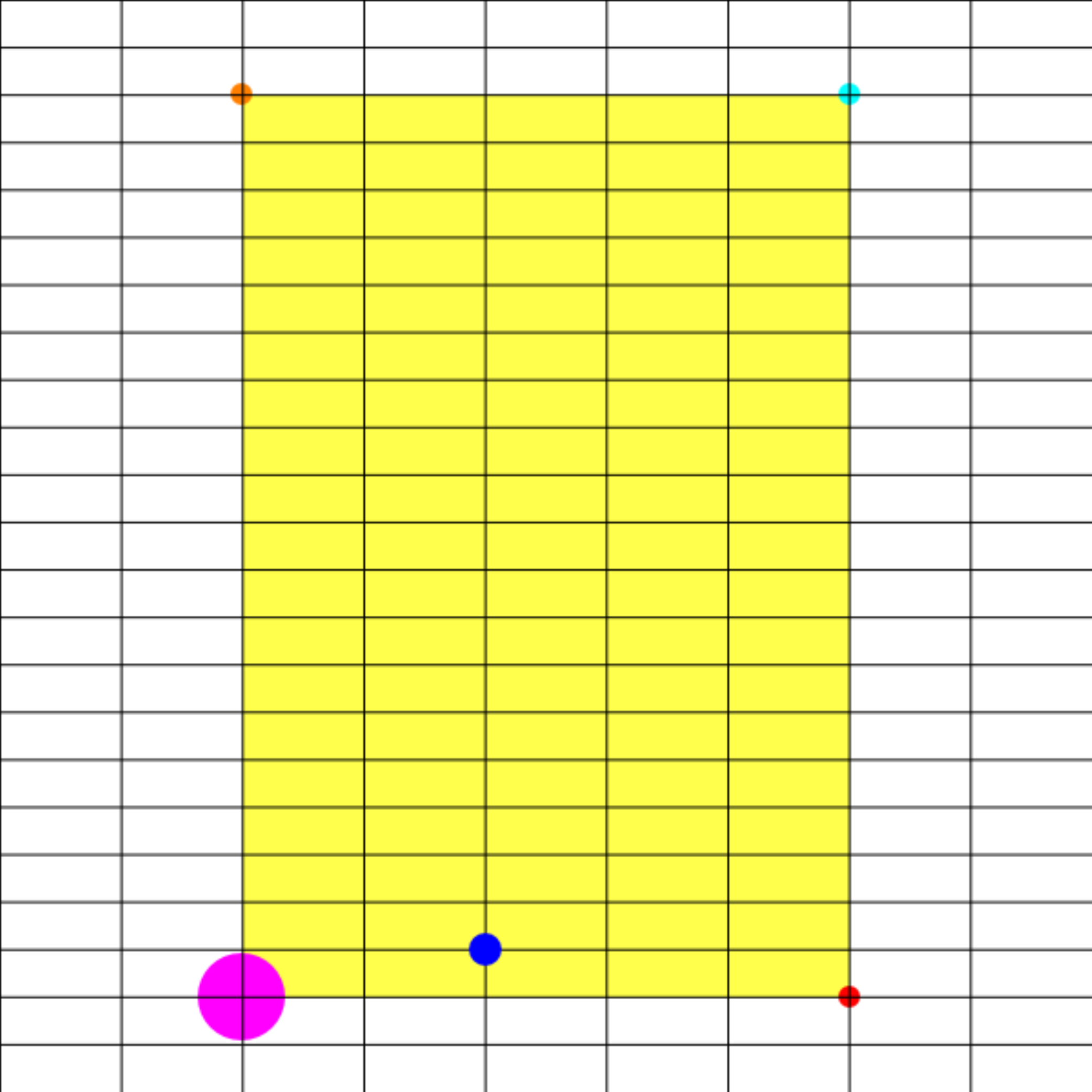}}
\caption{
The map $\phi$ is a bijection between the two finite
sets $\X=\Z^n/L$ and $\Y=\Z_{m_1} \times \cdots \times \Z_{m_n}$.
The picture visualizes the linear system
$4  x + 17 y = 2  \mod  5, 11 x + 13 y = 1  \mod 19$
which has the solution $(x,y)=(8,5)$. The vectors $(11,-2),(-2,9)$ span the
lattice of the kernel. 
}
\end{figure}

{\bf II. Construction of a solution} 
In order to construct a solution of $A \vec{x} = \vec{b}  \mod  \vec{m}$,
we have to find both the lattice $L_A$ and a particular solution
$\vec{x}$ of the equation $A\vec{x}=\vec{b}  \mod  \vec{m}$, 
then reduce $x$ modulo the lattice to make it small. \\

{\bf i) Finding a particular solution} \\
To find the particular solution, we pick {\bf Pivot elements} $a_{i j(k)}$
in the matrix $A$: these are entries in the $i$'th row which are relatively prime 
to $m_i$. Let $\vec{e}_j$ denote the standard basis in n-dimensional space.
Consider a curve $\vec{x}(t) = t \vec{e}_{j(1)}$ in $X$, where $t$ is an 
integer. Using the assumption on the rows, we see that there exists 
an integer $t_1$ so that $\vec{x}(t)$ solves the first equation. 
Now take the curve $\vec{x}(t) =  t_1 \vec{e}_{j(1)} + t m_1 \vec{e}_{j(2)}$. 
There is an integer $t_2$ so that $\vec{x}(t)$ solves the second equation.
We use here the fact that $m_1$ is relatively prime to $m_2$. Note that
$\vec{x}(t)$ solves the first equation for all $t$. 
Continue now until the final solution 
$\vec{x}(t) = \sum t_i (m_1 \cdots m_i) \vec{e}_{i j(i)}$ is found.

{\bf Remark:} Because $\X$ and $\Y$ are isomorphic groups, there is a one-dimensional 
``discrete line" $\vec{r}(t)=t \vec{v}$ such that $\vec{r}(t)/L_A$ covers  $\Y$. 
We could find a special solution by searching on that line, which is a problem of the 
CRT. We have the problem to find a vector $\vec{v}$ such that 
$A\vec{r}(t) = A (t \vec{v}) = t \vec{w}$ covers the entire set $\Y$.  \\

Lets look at the example
\begin{eqnarray*}
 4  x + 17 y &=& 2 \;  \mod \; 5  \\
 11 x + 13 y &=& 1 \;  \mod \; 19  \; . 
\end{eqnarray*}
Because all moduli are prime, any nonzero matrix element is a Pivot element in this example.
We can pick $j(1)=1, j(2)=2$. Take the line 
$\vec{x}(t) = t \vec{e}_1 = \left[ \begin{array}{c} t \\ 0 \end{array} \right]$ and look for
$t_1$ such that the first equation is solved.  This means $4 x = 2  \mod  5$ which 
gives $x=3$.  \\

Now consider the line $\vec{x}(t) = 3 \vec{e}_1  + 5 t \vec{e}_2 
= \left[ \begin{array}{c} 3  \\ 5t \end{array} \right]$. For every $t$, the first 
equation is solved. The second equation gives $33 + 65 t = 1  \mod  19$. 
which is solved by $t=15$. So, 
$\vec{x}(1) = \left[ \begin{array}{c} 3 \\ 75 \end{array}  \right]$ 
solves the system. \\

We could have solved the system also by taking the parametrized 
line $\vec{r}(t) = (x(t),y(t)=(t,t)$ which is mapped by $A$ to the 
line $(A \vec{r}(t)) = (11t,25t) =(t,5t)$ on the discrete torus. 
It leads to the CRT problem
\begin{eqnarray*}
 t    &=& 2 \;  \mod \; 5  \\
 5 t  &=& 1 \;  \mod \; 19 
\end{eqnarray*}
which is solved for $t=42$ so that we get the particular solution 
$(x,y)=\vec{r}(42)=(42,210)$. \\

{\bf ii) Finding the kernel}. \\
On every line $\vec{r}(t)=(0,...,t,...0)$, there is a point $\vec{x}$ which
solves $A \vec{x}=\vec{0}  \mod  \vec{m}$. By the pigeon hole principle, 
the set $\{ A \vec{x}   \mod  \vec{m}\; | \; t \in [0,M] \}$ must hit some 
point in the image twice. But then 
$A (\vec{x} - \vec{y})=\vec{0}  \mod  \vec{m}$.
If we take $n+1$ equations
$A \vec{x}^{(i)} = y^{(i)}  \mod  \vec{m}$, then the collection
of vectors $y^{(i)}$ is linearly dependent.  Therefore, there exist 
rational numbers $c_i$ such that 
$\sum_j c_j y^{(j)} = \vec{0}  \mod   \vec{m}$
so that $\sum_j c_j \vec{x}^{(j)}=\vec{0}$ is in the kernel.
After multiplying with a common multiple of the denominators of the 
rational numbers $c_j$, we can assume $c_j$ to be integers. 
We first look for $n$ linearly independent vectors $\vec{k}_i$ solving 
$A \vec{k}_i = \vec{0}  \mod  \vec{m}$. Define 
$K$ to be a matrix which contains the vectors $\vec{k}_i$ as row vectors.
Use the LLL algorithm (\cite{cohen} section 2.6)
to reduce the lattice to a small lattice. It turns
out that this is often not good enough. The lattice has a size which is a multiple
of $p$. In order to find the lattice $L_A$ of the kernel, we need
$$   {\rm det}(K) = M = m_1 m_2 \cdots m_n   \; . $$
Let $k={\rm det}(A)/p$ and let $k=q_1... q_l$ be the prime factorization of $k$. 
We can now look whether $\vec{y}^{(i)}/q_j$ are integer vectors in the kernel for 
each $i=1,\dots,n$ and $j=1,\dots,l$ and if yes replace the basis vectors.
Successive reduction of the lattice can lead us to the kernel for which
${\rm det}(K)=p$. If not, we start all over and construct a new lattice. 

\section{Outlook}

{\bf Complexity}. \\
For a linear system of equations $A \vec{x} = \vec{b}  \mod  \vec{m}$,
the problem is to find a maximal lattice $L_A$ in $\Z^n$, which is
the kernel of the group homomorphism $\vec{x} \mapsto A \vec{x}$ from
$\Z^n$ to the module $\Y=\Z_{m_1} \times \cdots \times \Z_{m_n}$
so that its fundamental region $\X$ is mapped bijectively onto
$A \X \subset \Y$. Next, we have to decide whether $\vec{b}$ is in
$A \X$ and if affirmative, construct $\vec{x} \in \X$
which satisfies $A \vec{x}=\vec{b}  \mod  \vec{m}$. How fast can this be done? \\
To find the kernel of the group homomorphism
$T(\vec{x}) = A \vec{x}  \mod  \vec{m}$, we produce
a large set of solutions of $T(\vec{x})=0$ and then
reduce this to a small lattice using the LLL algorithm. If $H$
is the matrix which contains the reduced kernel vectors as columns
then $A H = \vec{0}  \mod  \vec{m}$. In general,
${\rm det}(H) \neq M$, but we know that there exists a kernel for
which ${\rm det}(H) = M$. How do we find such a matrix $H$ directly? \\
To decide whether $A \vec{x} = \vec{b}  \mod  \vec{m}$ has a solution or not
is addressed in \cite{DoSt01}. The multivariable CRT gives a criterion for the
existence of solutions. One can often detect, whether one of the equations
has no solution. This happens for example, if $a_{i1},\dots ,a_{in},m_i$ have a common
denominator which is not shared by the denominators of $b_i$. If all $m_i$
are equal to some number $m$ with distinct prime factors can make a fast decision: by
the CRT, a solution exists if and only if a solution exists modulo
each prime factor of $m$ and the later decisions can be done by computing
determinants in finite fields. \\

{\bf Iteration of modular linear maps}. \\
The map $T(\vec{x}) = A \vec{x}  \mod  \vec{m}$ 
defines a dynamical system on the finite group $\Z_{m_1} \times ... \times \Z_{m_n}$. 
Since the discrete torus $\Y$ does not match with the torus 
$\X$, orbits on this finite set behave in general rather irregularly.
The system can be extended to the real torus 
$R/(m_1 \Z) \times R/(m_n Z)$, where it is in general a 
hyperbolic map. The orbits behave differently, if $A$ is very singular, for example
if $A$ has only one column.  The map
$$
T  \left[ \begin{array}{c} x \\ y \end{array} \right]
 = \left[ \begin{array}{c}  
     31 x + 34 y \\
      3 x + 38 y \\
   \end{array} \right]  \; \mod 
   \left[ \begin{array}{c} 7 \\ 17 \end{array} \right]  $$
for example has 6 different orbits on $\Y$ with a 
maximal orbit length of 49. It seems difficult to find
ergodic examples with different moduli where ergodic means
that there is only one orbit besides the trivial orbit of $\vec{0}=(0,0)$
a case which appears for example in
$$
T  \left[ \begin{array}{c} x \\ y \end{array} \right]
 = \left[ \begin{array}{c}  
     18 x +  5 y \\
      7 x + 14 y  \\
   \end{array} \right]  \; \mod 
   \left[ \begin{array}{c} 37 \\ 37 \end{array} \right]   \; . $$
  

{\bf Systems of modular polynomial equations} \\
The algorithm to solve systems of linear modular equations extends also 
to solve systems of nonlinear polynomial equations 
$\vec{P}(\vec{x}) = \vec{b}  \mod  \vec{m}$ with
$$ P_k(x_1,...,x_n) = b_k  \mod   m_k $$
too, but in general, we do not have criteria which assure that 
such a system has solution. We need to solve the individual equations
An example is Chevally's theorem (i.e. \cite{goldman})
which tells that $P$ is a polynomial of degree smaller than $n$ and zero
constant term, then $P(x_1,...,x_n) = p$ can be solved as long as $p$ is prime.
Lets look at the general problem.
Start solving the first equation. Using 
$\vec{x}=(a_{11} t,...,a_{1n} t)$ we have to solve a problem for a single variable
$q_1(t)=0  \;  \mod \;  m_1$, where $q_1$ is a polynomial.
With a solution $t_1$, try to solve the second equation for $t$ using
$\vec{x}=(m_1 a_{21} t,.., m_n a_{2n} t) + (a_{11} t_1,...,a_{1n} t_1)$.
which solves the first equation etc.  \\
For example, consider the system of nonlinear modular equations
\begin{eqnarray*}
      x^2+y^3+z^2   &=& 1  \mod  5  \\
      x^3+2 y^4-z^2 &=& 1  \mod  7 \\
      3x-2y^3+5 z^4 &=& 7  \mod  11  \; . 
\end{eqnarray*}
Start with the "Ansatz" 
$(x,y,z) = (t,t,t)$. The first equation is $t^2(2+t) = 1  \mod  5$
which has the solution $t=2$. Now put $(x,y,z) = (2,2,2) + t \cdot 5 (1,1,1)$.
which solves the first equation and plug it into the second equation.
This is $(2+5t)^2 (2+3t+t^2)=1 \; {\rm mod 7}$ and solved for $t=0$. 
The point $(2,2,2)+t(5,5,5)=(2,2,2)$ solves also the second equation. 
Now plug-in $(2,2,2)+ 5 \cdot 7(0,2t,t)$, which solves
the first two equations for all $t$, into the third equation which
requires to solve $6+5 (2+35 t)^4 - 2 (2+70 t)^3=7  \mod  11$
which is equivalent to $4+4t+2t^2+5 t^3+3 t^4=7  \mod  11$
and solved for $t=1$. So, the final solution found is 
$(2,2,2)+ 5 \cdot 7(0,2,1)=(2,72,37)$. This method does not necessarily find
small solutions like $(2,6,4)$.  \\

Nonlinear systems of modular equations with different moduli but with one
variable can be treated with the CRT. Ore \cite{ore}
illustrates it with the example
\begin{eqnarray*}
   x^3-2x+3&=&0   \mod  7 \\
   2x^2    &=&3   \mod  15 \; . 
\end{eqnarray*}
Because the first equation has solutions  $x=2  \mod  7$
and the second has solutions $x= \pm 3  \mod  15$, we are in the
case of the CRT. In general, systems of polynomial equations in one
variable often lead to CRT problems. 

\vfill

\pagebreak

\section{Mathematica code}

Here is some example code if a reader wants to experiment. 
The first few lines find and plot the lattice of solutions 
$A\vec{x}=0  \mod  \vec{p}$ by brute force and then do LLL 
reduction.

\lstset{language=Mathematica} \lstset{frameround=fttt}
\begin{lstlisting}[frame=single]
a = 13; b = 19; c = 11; d = 15; q = 29; p = 31;
s = {}; Do[If[Mod[a*x+b*y,p]==0 && Mod[c*x+d*y,q]==0, 
  s = Append[s, {x, y}]], {x, -100, 100}, {y, -100, 100}];
L=LatticeReduce[s]; M={{0,0},L[[1]]+L[[2]]};
Graphics[{{Blue,PointSize[0.01],Map[Point, s]},
  {Yellow,Polygon[{M[[1]],L[[1]],M[[2]],L[[2]]}]},
  {Red,PointSize[0.02],Map[Point,Join[L,M]]}}]
\end{lstlisting}

The following routines find solutions according to 
the proof of the multivariable CRT: 

\lstset{language=Mathematica} \lstset{frameround=fttt}
\begin{lstlisting}[frame=single]
Pivot[A_,P_]:=Module[{n=Length[A],p},p=Table[0,{n}]; 
 Do[Do[If[GCD[A[[i,j]],P[[i]]]==1,p[[i]]=j],{j,n}],{i,n}];p];
GCDv[p_]:=Max[Table[GCD[p[[i]],p[[j]]],
 {i,Length[p]},{j,i+1,Length[p]}]];
HasSol[A_,P_]:=Module[{p=PivotEntries[A, P]},
 Product[p[[i]],{i,Length[p]}]>0 && GCDv[P]==1];
CheckSol[A_,B_,X_,P_]:=
 Table[Mod[(A.X)[[i]]-B[[i]],P[[i]]],{i,Length[A]}];
LinearModSol[A_,B_,P_]:=Module[{n=Length[A],p,X,q,sum,j,pi},
 p=Pivot[A,P]; X=Table[0,{n}]; q=1; 
 Do[j=p[[i]]; pi=P[[i]]; bi=B[[i]]; aij=A[[i,j]];
  sum=Sum[A[[i,k]]*X[[k]],{k,n}];
 t=Mod[PowerMod[q*aij,-1,pi]*(bi-sum),pi];
 X[[j]]=X[[j]]+t*q; q=q*pi,{i,n}]; X];

A={{4,3,3,3},{1,-1,5,5},{1,5,3,7},{1,5,2,2}}; 
B={1,2,3,4}; P = {3,5,7,11}; X=LinearModSol[A,B,P]
CheckSol[A,B,X,P]
\end{lstlisting}

Finally, here is the verification of the example in the introduction

\lstset{language=Mathematica} \lstset{frameround=fttt}
\begin{lstlisting}[frame=single]
A={{101,107},{51,22}}; b={3,7}; m={117,71}; 
x={25,65}; L={{73,47},{-82,61}}; 
{Mod[A.x-b,m],Mod[A.L[[1]],m],Mod[A.L[[2]],m]}
\end{lstlisting}

\pagebreak

\bibliographystyle{plain}

\begin{thebibliography}{10}

\bibitem{Baumslag}
G.~Baumslag.
\newblock {\em Topics in combinatorial group theory}.
\newblock {Birkh\"auser} Verlag, 1993.

\bibitem{cohen}
H.~Cohen.
\newblock {\em A course in computational algebraic number theory}.
\newblock Graduate Texts in Mathematics. Springer, fourth printing edition,
  2000.

\bibitem{ConnorRobertsonSunzi}
J.J.~O' Connor and E.F. Robertson.
\newblock Sun {Zi}.
\newblock {http://www-history.mcs.st-andrews.ac.uk/Biographies/Sun\_Zi.html},
  2003.

\bibitem{Dauben}
J.W. Dauben.
\newblock Chinese mathematics.
\newblock In Victor Katz, editor, {\em Mathematics of
  Egypt,Mesopotamia,China,India and Islam, A sourcebook}. Princeton University
  Press, 2007.

\bibitem{dicksonII}
L.E. Dickson.
\newblock {\em History of the theory of numbers.{V}ol.{II}:{D}iophantine
  analysis}.
\newblock Chelsea Publishing Co., New York, 1966.

\bibitem{DoSt01}
A.~Dolzmann and T.~Sturm.
\newblock Parametric systems of linear congruences.
\newblock In {\em Computer algebra in scientific computing (Konstanz, 2001)},
  pages 149--166. Springer, Berlin, 2001.

\bibitem{Fiol}
M.A. Fiol.
\newblock Congruences in z, finite abelian groups and the chinese remainder
  theorem.
\newblock {\em Discrete Mathematics}, 67:101--105, 1987.

\bibitem{Gauss}
C.F. Gauss.
\newblock {\em Disquisitiones Arithmeticae}.
\newblock Lipsiae, in Comissis apud Gerh. Fleischer, Jun., 1801.

\bibitem{goldman}
J.~R. Goldman.
\newblock {\em The Queen of Mathematics, a historically motivated guide to
  number theory}.
\newblock A.K. Peters, Wellesley, Massachusetts, 1998.

\bibitem{GGR}
J.~Gopala~Krishna and K.~Raja~Rama Gandhi.
\newblock General, one \& several variable extentions of {C}hinese remainder
  theorem.
\newblock {\em Int. J. Math. Sci. Appl.}, 1(3):1201--1213, 1215--1223, 2011.

\bibitem{how01}
N.A. Howgrave-Graham and N.P. Smart.
\newblock Lattice attacks on digital signature schemes.
\newblock {\em Designs, Codes and Cryptography}, 23:283--290, 2001.

\bibitem{Hua}
L.K. Hua.
\newblock {\em Introduction to Number theory}.
\newblock Springer Verlag, Berlin, 1982.

\bibitem{Joyner}
D.~Joyner.
\newblock {\em Adventures in Group Theory, Rubik's Cube, Merlin's Machine and
  Other Mathematical Toys}.
\newblock Johns Hopkins University Press, second edition, 2008.

\bibitem{Katz2011}
V.J. Katz.
\newblock {\em A history of Mathematics}.
\newblock Addison-Wesley, second edition, 2011.

\bibitem{Chinchin92}
A.Ya. Khinchin.
\newblock {\em Continued Fractions}.
\newblock Dover, third edition, 1992.

\bibitem{Koshy}
T.~Koshy.
\newblock {\em Elementary Number Theory with Applications}.
\newblock Elsevier, 2.nd edition, 2007.

\bibitem{Libbrecht}
U.~Libbrecht.
\newblock {\em Chinese Mathematics in the Thirteenth Century}.
\newblock MIT Press, 1973.

\bibitem{Martzloff}
J-C. Martzloff.
\newblock {\em A history of Chinese Mathematics}.
\newblock Springer Verlag, second edition, 1997.

\bibitem{mathews1892}
G.B. Mathews.
\newblock {\em Theory of Numbers}.
\newblock Cambridge, Deighton, Bell and Co.; London, G. Bell and Sons, 1892.

\bibitem{Nicomachus}
Nicomachus of~Gerasa.
\newblock {\em Introduction to Arithmetic, translated by M.L. D'Oodge}.
\newblock McMilloan Company, 100 AC est.
\newblock http://www.archive.org/details/NicomachusIntroToArithmetic.

\bibitem{ore}
O.~Ore.
\newblock On the averages of the divisors of a number.
\newblock {\em Amer. Math. Monthly}, 55:615--619, 1948.

\bibitem{sheng88}
Kang~Sheng Shen.
\newblock Historical development of the {C}hinese remainder theorem.
\newblock {\em Arch. Hist. Exact Sci.}, 38(4):285--305, 1988.

\bibitem{Kangsheng}
S.~Kang Sheng.
\newblock Historical development of the {C}hinese remainder theorem.
\newblock {\em Arch. Hist. Exact Sci.}, 38(4):285--305, 1988.

\bibitem{Stillwell}
J.~Stillwell.
\newblock {\em Mathematics and its History}.
\newblock Undergraduate Texts in Mathematics. Springer, 1989.

\bibitem{Swetz}
F.~Swetz.
\newblock The evolution of mathematics in ancient china.
\newblock {\em Mathematics Magazine}, 52:10--19, 1979.

\bibitem{Andreescu}
I.~Cucurezeanu T.~Andreescu, D.~Andrica.
\newblock {\em An Introduction to Diophantine Equations}.
\newblock {Birkh\"auser}, 2010.

\bibitem{YanShiran}
L.~Yan and D.~Shiran.
\newblock {\em Chinese Mathematics, A concise history}.
\newblock Oxford Science Publications. Clarendon Press, 1987.
\newblock Translated by J.N. Crosley and A.W. Lun.

\end{thebibliography}

\end{document}